\documentclass[reqno]{amsproc}
\usepackage{amssymb,amsmath}
\usepackage{amsfonts}
\usepackage{amsthm}
\usepackage{longtable}
\usepackage{graphicx}
\usepackage{wrapfig}
\usepackage{url}
\usepackage[colorlinks,citecolor=red,urlcolor=blue,bookmarks=false,hypertexnames=true]{hyperref}
\usepackage{mathtools}

\usepackage{a4}

\newtheorem{prop}{Proposition}
\newtheorem{example}{Example}
\newtheorem{theorem}{Theorem}
\newtheorem{lemma}{Lemma}
\newtheorem{cor}{Corollary}

\newcommand{\Irr}{\textnormal{Irr}}
\newcommand{\cd}{\textnormal{cd}}
\newcommand{\nl}{\textnormal{Nl}}
\newcommand{\lin}{\textnormal{Lin}}

\newcommand{\Core}{\textnormal{Core}}

\newcommand{\ind}{\mathord\uparrow} 

\newcommand{\irr}{\textnormal{irr}}
\newcommand{\GL}{\textnormal{GL}}

\title{On the representation dimension of finite $p$-groups}
\author{Gurleen Kaur}
\address{Department of Mathematics, Indian Institute of Technology Ropar, Punjab 140001, India}
\email{gurleenkaur992gk@gmail.com}
\author{Amit Kulshrestha}
\address{Indian Institute of Science Education and Research Mohali, Knowledge City, Sector 81, Mohali 140306, India}
\email{amitk@iisermohali.ac.in}
\author{Ayush Udeep}
\address{SRM Institute of Science and Technology, Kattankulathur campus, SRM Nagar, Chennai 603203, India}
\email{udeepayush@gmail.com}

\begin{document}

\maketitle

\begin{abstract}
    For a finite group $G$, the representation dimension $\delta(G)$ is the least dimension of a faithful complex representation of $G$. We prove that $\delta(H\times K) = \delta(H) + \delta(K)$ whenever $H$ and $K$ are $p$-groups for a fixed prime $p$, and show by example that this additivity can fail more generally for nilpotent groups. We also determine $\delta(G)$ for several important classes of non-abelian $p$-groups, {\it viz.} VZ groups, Camina groups, and metacyclic groups; and compute $\delta(G)$ for all groups of order $p^6$ ($p\geq 5$), organized by their isoclinism family.
\end{abstract}

\section{Introduction}

Throughout this article, all groups are finite, and all representations are over the complex field. The \emph{representation dimension} of a finite group $G$, denoted by $\delta(G)$, is the least dimension of a faithful representation of $G$ (note that the regular representation ensures the existence of a faithful representation of $G$).  In other words, $\delta(G)$ is the smallest positive integer $n$ for which $G$ is embedded into $\GL_{n}(\mathbb{C})$. Computing the representation dimension is one of the classical problems in the representation theory of finite groups. Its investigation has also led to its link with the essential dimension of $G$ (see \cite{KM08} for further details).
In this direction, a related quantity is the minimal dimension of a faithful irreducible representation of $G$, denoted by $\delta_{\irr}(G)$. 
Note that a group $G$ may not have a faithful irreducible representation, and hence $\delta_{\irr}(G)$ may not exist. For example, $\delta_{\irr}(C_{2} \times C_{2})$ does not exist. However, if $\delta_{\irr}(G)$ exists, then $\delta(G) \leq \delta_{\irr}(G)$. Another quantity, the minimal dimension of a faithful complex projective representation
of $G$, has been well studied for simple groups $G$ (see \cite{TZ00}).

Investigation of minimal dimension of a faithful representation of $p$-groups was started by Meyer and Reichstein \cite{MR10}. Note that we denote primes by $p$ or $q$. As a consequence of \cite[Theorem 1.2]{MR10}, if $G$ is a $p$-group and $\rho$ is a faithful representation of $G$ of least dimension, then $\rho$ decomposes into exactly $r$ 
irreducible representations of $G$, where $r$ denotes the rank of the center of $G$ 
(Bardestani et al. provided an alternative approach in \cite{BMS16}). Later, Cernele et al. \cite{CKR11} coined the term ``representation dimension" and computed an upper bound for the representation dimension of a $p$-group.
In \cite{KKS23}, Kaur et al. computed the representation dimension for some finite $p$-groups, in particular, for all $p$-groups of order up to $p^{5}$, for
direct products of such groups with certain conditions, and for groups with certain conditions on nonlinear irreducible characters.

In this article, we investigate the representation dimension of the direct product of various classes of groups. We begin with the following result. For notational convenience, $\cd(G)$ denotes the set of all the irreducible character degrees of $G$. 
Note that a group is called \emph{normally monomial} if all of its irreducible characters are induced from linear characters of some normal subgroups. Normally monomial groups covers a large class of solvable groups, in particular, metabelian groups \cite{B69}. 
\begin{theorem} \label{thm:HtimesK}
    For a fixed prime $p$, let $H$ and $K$ be normally monomial $p$-groups with cyclic centers and let $\psi$ be a faithful character of least degree of $H\times K$. Then, $\psi$ must decompose as $\chi + \eta$, for some irreducible characters $\chi, \eta$ of $H\times K$ with $\chi(1) = \max\cd(H)$ and $\eta(1) = \max\cd(K)$.
\end{theorem}
Further, we prove that for a fixed prime $p$, if $H$ and $K$ are non-abelian $p$-groups, then $\delta(H\times K) = \delta(H)+\delta(K)$ (see Theorem \ref{thm:HtimesKpgroup}). We also provide counterexamples that imply that if $H$ and $K$ are nilpotent, then $\delta(H\times K)$ may not be equal to $\delta(H)+\delta(K)$. Interestingly, if $K$ is a cyclic group and $H$ is either a non-abelian $p$-group with cyclic center and $\cd(H) = \{ 1, p^a \}$ or $\{ 1, p, p^b \}$, or a normally monomial $p$-group with cyclic center, then $\delta(H\times K) = \delta(H) + \delta(K)$ (see Theorems \ref{thm:HtimesA1} and \ref{thm:HtimesA2}).
We also compute the representation dimension for some important classes of non-abelian $p$-groups, namely VZ $p$-groups, Camina $p$-groups, metacyclic $p$-groups and groups of order $p^6$ ($p\geq 5$ is a prime) (see Propositions \ref{prop:VZgroup}--\ref{prop:metacyclic} and Table \ref{t:delta G p6}.
Representation dimension for groups of order $2^6$ and $3^6$ can be obtained easily by using the {\sf GAP} \cite{GAP} function {\tt EmbeddingDegree} developed by Kaur et al. \cite{KKS23}.

\section{Preliminaries} \label{sec:preliminaries}
In this section, we summarize notations and some background results that we use throughout the article.

For a group $G$, $\mathcal{Z}(G)$ and $G'$ denote the center and the commutator subgroup of $G$, respectively. By $d(G)$, we denote the minimal number of generators of $G$. We denote the trivial element or trivial subgroup of $G$ by 1, depending on the context. For a subgroup $H$ of $G$, the core of $H$ in $G$, denoted by $\Core_{G}(H)$, is the largest normal subgroup of $G$ contained in $H$. For a character $\chi$ of $G$, $\ker(\chi)$ denotes the subgroup $\{ g\in G: \chi(g) = \chi(1) \}$ of $G$. 
By $\Irr(G), \lin(G)$ and $\nl(G)$, we denote the set of all the irreducible characters, the set of all the linear characters, and the set of all the non-linear irreducible char
acters of $G$, respectively. For a subgroup $H$ of $G$, and an irreducible character $\lambda$ of $H$, the induced character of $\lambda$ from $H$ to $G$ is denoted by $\lambda\ind_{H}^{G}$. 

It is known that two complex representations are isomorphic if and only if the characters afforded by them are equal. Hence, to study the representation dimension, we often take the help of characters. Let $\chi$ be a faithful character of $G$ of least degree, that is, $\chi(1)=\delta(G)$. Then, $\chi=\chi_{1}+\chi_{2}+\cdots+\chi_{m},$ for some distinct irreducible characters $\chi_{i}$'s of $G$. Denote the set $\{ \chi_{1},\chi_{2},\cdots,\chi_{m} \}$ by $\mathcal{X}$. Then,
\begin{equation} \label{eq:SetXProperty}
    \bigcap_{\chi \in \mathcal{X}}\ker(\chi) = 1 \text{ and } \bigcap_{\chi \in Y}\ker(\chi) \neq 1 \text{ for all } Y \subsetneq \mathcal{X}.
\end{equation}
In this case, we say that $\mathcal{X}$ {\it affords} $\delta(G)$.
In \cite{BMS16}, Bardestani et al. obtained the following result about the cardinality of such a set $\mathcal{X}$ when $G$ is a $p$-group.   
 \begin{lemma} \textnormal{\cite[Lemma 3.5]{BMS16}} \label{lemma:Bardestanietal}
    Let $G$ be a non-abelian $p$-group. Then, $|\mathcal{X}| = d(\mathcal{Z}(G))$, for any $\mathcal{X}\subset \Irr(G)$ affording $\delta(G)$.
\end{lemma}

In \cite[Lemma 3.4]{M21}, Moret\'{o} proved that $\delta(A) = d(A)$ for any abelian group $A$. In \cite{KKS23}, Kaur et al. provided examples of the representation dimension of numerous groups and computed the representation dimension for various classes of non-abelian $p$-groups, such as $p$-groups of order $p^n$, where $n \leq 5$, normally monomial $p$-groups with cyclic center, and others.  In \cite{M99}, Mann proved the following result which will be used in the subsequent sections.

\begin{lemma} \textnormal{\cite[Proposition 3]{M99}} \label{lemma:MannNormallyMonomial}
    Let $G$ be a normally monomial group and $\chi$ be a faithful irreducible character of $G$. If $\chi$ is induced from a linear character of some normal subgroup $N$ of $G$, then $N$ must be abelian. Further, all such normal subgroups $N$ have maximum order among all the abelian subgroups of $G$, and $\chi(1) = \max \cd(G)$.
\end{lemma}



Now, Lemma \ref{lemma:deltaHtimesA} provides an upper bound on the representation dimension of the direct product of two groups.
\begin{lemma} \label{lemma:deltaHtimesA}
    Let $H$ and $K$ be groups. Then
    \[ \delta(H \times K) \leq \delta(H) + \delta(K). \]
\end{lemma}
\begin{proof}
Let $\chi_1$ and $\chi_2$ be faithful characters of $H$ and $K$ respectively satisfying $\chi(1) = \delta(H)$ and $\chi_2(1) = \delta(K)$. Then $\chi_1 1_K$ and $1_H \chi_2$ are faithful characters of $H \times K$. Let $\chi = \chi_1 1_K + 1_H \chi_2$. Since $\ker(\chi) = \ker(\chi_1 1_K) \cap \ker(1_H \chi_2)$, $\chi$ is faithful as well. Therefore,
\[ \delta(H \times K) \leq \chi(1) = \chi_1(1) + \chi_2(1) = \delta(H) + \delta(K), \]
which completes the proof.
\end{proof}

In view of Lemma \ref{lemma:deltaHtimesA}, it is natural to investigate the relation between $\delta(H\times K)$ and $ \delta(H) + \delta(K)$ for various classes of groups. In Section \ref{sec:directproduct}, we prove that equality holds, for certain classes of groups, in the statement of Lemma \ref{lemma:deltaHtimesA}. 

\section{Representation dimension of the direct product of groups} \label{sec:directproduct}

In this section, we explore the representation dimension of the direct product of two groups. We study the irreducible constituents of a faithful character of least degree of $H\times K$, where $H$ and $K$ are normally monomial $p$-groups with cyclic centers. Further, we prove that the representation dimension of the direct product of $p$-groups 
coincides with the sum of the representation dimension of their underlying groups. However, the same is not true for nilpotent groups and we provide several counterexamples to support our claim. We begin with the following lemma.

\begin{lemma} \label{lemma:HtimesK}
    If $H$ is a normally monomial $p$-group with cyclic center and  $\varphi$ is an irreducible character of $H$ whose degree is strictly less than the maximum degree of all the irreducible characters of $H$, then it must kill all central $p$-elements of $H$.
\end{lemma}
\begin{proof}
    Let $p^b = \max \cd(H)$. Since $H$ is normally monomial, there exists $M \trianglelefteq H$ such that $\varphi = \lambda\ind_{M}^{H}$, for some $\lambda\in \lin(M)$. Now, $\varphi(1)< p^b$ implies that $|H/M| < p^b$, therefore $|M| > |H|/p^b$. By Lemma \ref{lemma:MannNormallyMonomial}, the maximum order of an abelian subgroup of $H$ is $|H|/p^b$ which implies that $M$ is a non-abelian subgroup of $H$. Since $M'$ is a characteristic subgroup of $M$, and $M\trianglelefteq H$, we get $M' \trianglelefteq H$. Since $H$ is a $p$-group and $M' \trianglelefteq H$, therefore $1 \neq \mathcal{Z}(H) \cap M' \subset \ker(\lambda)$. Consequently, $\mathcal{Z}(H) \cap M' \subset \Core_{H} \left( \ker(\lambda) \right) = \ker(\varphi)$. Since $\mathcal{Z}(H)$ is cyclic, all the central elements of order $p$ are contained in $\ker(\varphi)$.
\end{proof}
Now we prove Theorem \ref{thm:HtimesK} with the help of Lemma \ref{lemma:deltaHtimesA} and \ref{lemma:HtimesK}.\\

\noindent{\bf Proof of Theorem \ref{thm:HtimesK}.}
    Let $\max\cd(H)$ and $\max\cd(K)$ be $p^b$ and $p^m$, respectively. Suppose $G = H\times K$. Then $\max \cd(G) = p^{b+m}$. 
    By Lemma \ref{lemma:MannNormallyMonomial} and \ref{lemma:deltaHtimesA}, we get 
    \begin{equation} \label{eq:HtimesK}
        \delta(G) \leq \delta(H) + \delta(K) = p^b + p^m.
    \end{equation}
    Here $\mathcal{Z}(G) = \mathcal{Z}(H) \times \mathcal{Z}(K)$, and thus, $d(\mathcal{Z}(G)) = 2$. Suppose $\mathcal{X} \subset \Irr(G)$ affording $\delta(G)$. By Lemma \ref{lemma:Bardestanietal}, $|\mathcal{X}| = 2$, and let us suppose that $\mathcal{X} = \{ \chi, \eta \}$. First, we list all possibilities for the degrees of $\chi$ and $\eta$, and then we eliminate all those possibilities that contradict Equation \eqref{eq:SetXProperty}.
    The possibilities for $\chi(1)$ and $\eta(1)$ are the following:
    \begin{enumerate}
        \item [(i)] $\chi(1) > p^b$ and $\eta(1) > p^m$;
        \item [(ii)] $\chi(1) > p^b$ and $\eta(1) = p^m$;
        \item [(iii)] $\chi(1) = p^b$ and $\eta(1) > p^m$;
        \item [(iv)] $\chi(1) = p^b$ and $\eta(1) < p^m$;
        \item [(v)] $\chi(1) < p^b$ and $\eta(1) < p^m$;
        \item [(vi)] $\chi(1) < p^b$ and $\eta(1) = p^m$;
        \item [(vii)] $\chi(1) < p^b$ and $\eta(1) > p^m$;
        \item [(viii)] $\chi(1) > p^b$ and $\eta(1) < p^m$;
        \item [(ix)] $\chi(1) = p^b$ and $\eta(1) = p^m$.
    \end{enumerate}
     \noindent {\bf Cases (i), (ii), and (iii):} Cases (i), (ii), and (iii) imply that $\delta(G) = \chi(1) + \eta(1) > p^b + p^m$, which contradicts Equation \eqref{eq:HtimesK}. Thus, Cases (i), (ii), and (iii) are not possible.\\
     Before we move further, we need a few more notations. Let $\chi = \chi_1 \chi_2$ and $\eta = \eta_1 \eta_2$, for some $\chi_1, \eta_1 \in \Irr(H)$ and $\chi_2, \eta_2 \in \Irr(K)$. Since $H$ and $K$ are normally monomial $p$-groups, there exist $M_1, N_1 \trianglelefteq H$ and $M_2, N_2 \trianglelefteq K$ such that 
     \[ \chi_{1} = \lambda_{1}\ind_{M_1}^{H},~ \eta_{1} = \psi_{1}\ind_{N_1}^{H},~ \chi_{2} = \lambda_{2}\ind_{M_2}^{K}, \text{ and } \eta_{2} = \psi_{2}\ind_{N_2}^{K}, \]
     for some $\lambda_i \in \lin(M_i)$ and some $\psi_i \in \lin(N_i),$ where $i=1,2$.\\
     


     
     \noindent {\bf Case (iv):} Here, we have $\chi(1) = p^b$ and $\eta(1) < p^m$. Now, $p^m > \eta(1) = \eta_1(1) \eta_2(1)$, and thus we obtain $\eta_2(1) < p^m$. Then, from Lemma \ref{lemma:HtimesK}, all the elements of order $p$ of $\mathcal{Z}(K)$ belong to $\ker(\eta_2) \subset \ker(\eta)$.
     Now, we obtain the implications of $\chi(1) = p^b$. \\
     First, consider the case when $b< m$. Here, $\chi_1(1)\chi_2(1) = \chi(1) = p^b < p^m \Rightarrow \chi_2(1) < p^m$. Then, again from Lemma \ref{lemma:HtimesK}, all the elements of order $p$ of $\mathcal{Z}(K)$ belong to $\ker(\chi_2) \subset \ker(\chi)$. Therefore, $\ker(\chi) \cap \ker(\eta) \neq 1$, which is a contradiction.\\
     Now, we consider the case when $b \geq m$. There are two possibilities:
     \begin{enumerate}
         \item [(a)]  $\chi_1(1) = p^b$ and $\chi_2(1) = 1$: In this case, $\chi_2 \in \lin(K)$. Then, $1\neq \mathcal{Z}(K) \cap K' \subseteq K' \subset \ker(\chi)$, which implies that $\ker(\chi) \cap \ker(\eta) \neq 1$, a contradiction.
         \item [(b)] $\chi_1(1) = p^{b-m}$ and $\chi_2(1) = p^m$, where $1\leq m\leq b$: This case is only possible if $p^{b-m} \in \cd(H)$. In this case, $\chi_1(1) < p^b$. Then, from Lemma \ref{lemma:HtimesK}, all the elements of order $p$ of $\mathcal{Z}(H)$ belong to $\ker(\chi_1) \subset \ker(\chi)$.
 Further, $\eta_{1} < p^{b}$ since $\eta_1(1) \eta_2(1) = \eta(1) < p^m \leq p^b$. Again,  from Lemma \ref{lemma:HtimesK}, all the elements of order $p$ of $\mathcal{Z}(H)$ belong to $\ker(\eta_1) \subset \ker(\eta)$. Since $\mathcal{Z}(H)$ is cyclic, we get $\ker(\chi) \cap \ker(\eta) \neq 1$, which is a contradiction.
     \end{enumerate}
     Therefore, Case (iv) is not possible.\\
     
     \noindent {\bf Case (v,vi):} One can verify that Case (v) and Case (vi) are not possible, based on the similar arguments as provided in Case (iv).\\ 

     \noindent {\bf Case (vii):} Here, we have $\chi(1) < p^b$ and $\eta(1) > p^m$. First, if $b< m$, then $\eta(1) > p^m$ implies that $\eta(1) \geq p^{m+a_1}$ for some $a_1 \geq 1$. Since $\chi(1) \geq 1$, we get $\chi(1) + \eta(1) \geq p^{m+a_1} + 1$. Suppose, on the contrary, that $p^{m+a_1} + 1 \leq p^b + p^m$. This implies that
     \[ p^b + p^m - p^{m+a_1} \geq 1 \Rightarrow p^b[ 1 + p^{m-b}(1 - p^{a_1}) ] \geq 1, \]
     which is not possible since $1-p^{a_1} \leq -1$. Therefore, $\chi(1) + \eta(1) \geq p^{m+a_1} + 1 > p^b + p^m$, which is a contradiction from Equation \eqref{eq:HtimesK}.\\
     Now, let us consider the case when $b>m$. Let $b = j + m$ for some natural number $j>0$. We have the following possibilities:
     \begin{enumerate}
         \item [(a)] $\chi(1) < p^b$ and $p^m < \eta(1) < p^b$:  In this case, we have $\chi_1(1) < p^b$ and $\eta_1(1) < p^b$. From Lemma \ref{lemma:HtimesK}, all the elements of order $p$ of $\mathcal{Z}(H)$ belong to both $\ker(\chi)$ and $\ker(\eta)$. Since $\mathcal{Z}(H)$ is cyclic, we get $\ker(\chi) \cap \ker(\eta) \neq 1$, a contradiction.
         \item [(b)] $\chi(1) < p^b$ and $\eta(1) =p^b$: In this case, we have three further possibilities: either $p^m < \chi(1) < p^b$, or $\chi(1) = p^m$, or $\chi(1) < p^m$. If $p^m < \chi(1) < p^b$ and $\eta(1) = p^b$, then $\chi(1) + \eta(1) > p^b + p^m$, a contradiction by Equation \eqref{eq:HtimesK}. \\
         Further, if $\chi(1) = p^m$ and $\eta(1) = p^b$, then the case is the same as Case (ix).\\
         Otherwise, we have $\chi(1) < p^m$ and $\eta(1) = p^b$. Then, the proof follows on similar lines as Case (iv), and we have proved above that Case (iv) is not possible.
         \item [(c)] $\chi(1) < p^b$ and $\eta(1) > p^b$: In this case, $\eta(1) \geq p^{m+j+1}$. Since $\chi(1) \geq 1$, we get $\chi(1) + \eta(1) \geq p^{m+j+1} + 1$. Note that $p^{m+j+1} + 1 \leq p^b + p^m$ implies that 
     \[ p^b + p^m - p^{m+j+1} \geq 1 \Rightarrow p^m[ 1 - p^{j}(p - 1) ] \geq 1, \]
     which is not possible since $1-p^{j}(p-1) \leq -1$. Therefore, $\chi(1) + \eta(1) \geq p^{m+j+1} + 1 > p^b + p^m$, which is a contradiction from Equation \eqref{eq:HtimesK}.
     \end{enumerate}
     Finally, if $b = m$, then we have $\chi(1) < p^b$ and $\eta(1) > p^b$. Then, $\eta(1) > p^b$ implies that $\eta(1) \geq p^{b+a_2}$ for some $a_2 \geq 1$. Hence, $\chi(1) + \eta(1) \geq p^{b+a_2} + 1$. It is easy to see that $p^{b+a_2} + 1 > 2p^b$. Hence, $\chi(1) + \eta(1) > 2p^b$, which is a contradiction from Equation \eqref{eq:HtimesK}.\\
     Therefore, Case (vii) is not possible, except when $b>m$, $\chi(1) = p^m$, and $\eta(1) = p^b$, which is a sub-case of Case (ix).\\

     \noindent {\bf Case (viii):} Here, we have $\chi(1) > p^b$ and $\eta(1) < p^m$. First, we consider the case when $b> m$. Let $b = m + j$ for some natural number $j>0$. Then, $\chi(1) \geq p^{m+j+1}$ and $\eta(1) \geq 1$, so that we get $\chi(1) + \eta(1) \geq p^{m+j+1} + 1$. Then, similar to the sub-case (c) in Case (vii), we get $\chi(1) + \eta(1) > p^b + p^m$, which is a contradiction from Equation \eqref{eq:HtimesK}.\\
     Further, if $b=m$, the case is the same as the sub-case $b = m$ in Case (vii).\\
     Otherwise, if $b<m$, then $\chi(1) \geq p^b$ implies that $\chi(1) \geq p^m$. Then, we have the following possibilities:
     \begin{enumerate}
         \item [(a)] $\chi(1) = p^m$ and $\eta(1) = p^b$: This case is the same as Case (ix).
         \item [(b)] $\chi(1) > p^m$ and $\eta(1) = p^b$: This case is a sub-case of Case (iii), and we have proved above that Case (iii) is not possible.
         \item [(c)] $\chi(1) = p^m$ and $\eta(1) < p^b$: This case is a sub-case of Case (vi), and we have proved above that Case (vi) is not possible.
         \item [(d)] $\chi(1) > p^m$ and $\eta(1) < p^b$: This case is a sub-case of Case (vii), and we have proved above that Case (vii) is not possible.
     \end{enumerate}
     Therefore, Case (viii) is not possible, except when $b<m$, $\chi(1) = p^m$, and $\eta(1) = p^b$, which is a sub-case of Case (ix).\\

     \noindent We are now left with only Case (ix), that is, $\chi(1) = p^b$ and $\eta(1) = p^m$. We summarize our discussion in the following table: 
     
     \begin{longtable}[c]{| l| c |c|c|c|}
	\caption*{ \label{t}}\\

	\hline
	Case &  $\chi(1)$ & $\eta(1)$ & Outcome & Reason \\ 
	\hline
	\endfirsthead
	
	\hline
	\multicolumn{2}{|c|}{Continuation of Table \ref{t:p34}}\\
	\hline
	Case &  $\chi(1)$ & $\eta(1)$ & Outcome & Reason \\  
	\hline
	\endhead
	
	\hline
	\endfoot

	\endlastfoot
	\hline
	
	(i), (ii), (iii) &  $>p^{b}$ & $>p^{m}$ & Contradiction & Violates upper bound (Equation \eqref{eq:HtimesK}) \\
    (iv) &  $=p^{b}$ & $<p^{m}$ & Contradiction &   Forced share central $p$-element in kernels \\ 
    (v), (vi) &  $<p^{b}$ & $\leq p^{m}$ & Contradiction & Forced share central $p$-element in kernels \\	(vii), (viii) &  $<p^{b}$ & $>p^{m}$ & Contradiction & Violates upper bound (Equation \eqref{eq:HtimesK}) \\ (ix) &  $=p^{b}$ & $=p^{m}$ & Feasible case &  \\

		\hline

\end{longtable} This concludes the proof.\qed
\\

Theorem \ref{thm:HtimesK} implies that for normally monomial $p$-groups $H$ and $K$ with cyclic center, we get $\delta(H\times K) = \delta(H) + \delta(K)$. In the next theorem, we prove that for any $p$-groups $H$ and $K$, $\delta(H\times K) = \delta(H) + \delta(K)$. The proof utilizes the techniques of \cite[Theorem 2]{W75}.
\begin{theorem} \label{thm:HtimesKpgroup}
    Let $H$ and $K$ be $p$-groups, for a fixed prime $p$. Then, $\delta(H\times K) = \delta(H) + \delta(K)$.
\end{theorem}
\begin{proof}
    Let $d(\mathcal{Z}(H)) = m$, $d(\mathcal{Z}(K)) = n$, and let $G = H\times K$. Then, $d(\mathcal{Z}(G)) = m+n$. Let $\chi$ be a faithful character of $G$ of least degree. From Lemma \ref{lemma:Bardestanietal}, $\chi$ decomposes as a sum of $m+n$ distinct irreducible characters of $G$, say, $\chi_{1}, \chi_{2}, \cdots, \chi_{m+n}.$
   Then, 
    \[ \bigcap_{i=1}^{m+n} \ker(\chi_i) = 1 \text{ and } \bigcap_{i=1, i\neq j}^{m+n} \ker(\chi_i) \neq 1 \text{ for each } 1\leq j \leq m+n. \]
    Let $T_j = \cap_{i=1, i\neq j}^{m+n} \ker(\chi_i)$ for each $1\leq j \leq m+n$. By \cite[Lemma, p. 899]{W75}, $1\neq T_j \cap \mathcal{Z}(G)$ is cyclic for each $j$. Let $\Omega_{1}(\mathcal{Z}(G))$ denote the subgroup generated by the elements of $\mathcal{Z}(G)$ of order $p$. Note that the total number of elements of order $p$ in an abelian group $A$ of rank $k$ equals $\sum_{r=0}^{k-1} \binom{k}{r}(p-1)^{k-r}$, i.e. the number depends on the rank of $A$. Hence, the rank of $\Omega_{1}(\mathcal{Z}(G))$ must be equal to the rank of $\mathcal{Z}(G)$, i.e. \[ \Omega_{1}(\mathcal{Z}(G)) \cong C_{p}^{m+n}. \] 
    Then, $T_j \cap \Omega_{1}(\mathcal{Z}(G))$ is also a nontrivial cyclic group for each $1\leq j \leq m+n$. Suppose, for each $j$, $T_j \cap \Omega_{1}(\mathcal{Z}(G)) = \langle t_j \rangle$ for some $t_j \in \Omega_{1}(\mathcal{Z}(G))$. 
    Then, 
    $t_i \in \ker(\chi_j)$ for each $i\neq j$, where $1\leq i,j \leq m+n$. Clearly,
    \begin{equation} \label{eq:Sj}
        S_j = \langle t_i : 1\leq i \leq m+n, ~i\neq j \rangle \subseteq \ker(\chi_j) \cap \Omega_{1}(\mathcal{Z}(G)).
    \end{equation}
    \noindent {\bf Claim 1:} $S_j \cong C_{p}^{m+n-1}$ for each $1\leq j \leq m+n$.\\
    Since $t_j \in \Omega_{1}(\mathcal{Z}(G))$, the order of $t_j$ is $p$, for each $j$. To prove the claim, we only need to show that $t_j \notin S_j$ for each $j$. This would imply that for each $j$, $t_j$ cannot be expressed as a product of $t_1, \ldots, t_{j-1}, t_{j+1}, \ldots, t_{m+n}$, and hence, $S_j$ would prove to be an elementary abelian group of order $p^{m+n-1}$. \\
    Suppose, on the contrary, that $t_k \in S_k$ for some $1\leq k \leq m+n$. Then,
    \[ \langle t_k \rangle = T_k \cap \Omega_{1}(\mathcal{Z}(G)) = \left( \bigcap_{i=1, i\neq k}^{m+n} \ker(\chi_i) \right) \cap \Omega_{1}(\mathcal{Z}(G)). \]
    On the other hand, $S_k \subseteq \ker(\chi_{k}) \cap \Omega_{1}(\mathcal{Z}(G))$ (from Equation \eqref{eq:Sj}). Since $t_k \in S_k$, we get $1\neq t_k \in \cap_{i=1}^{m+n}\ker(\chi_i)$, which is a contradiction. Hence $t_j \notin S_j$ for each $j$, and thus, Claim 1 is proved.\\
    Now, if $\ker(\chi_{l}) \cap \Omega_{1}(\mathcal{Z}(G))$ properly contains $S_l$ for some $l$, then $\ker(\chi_{l}) \cap \Omega_{1}(\mathcal{Z}(G))$ must be isomorphic to the elementary abelian group of rank $m+n$. This implies that $\ker(\chi_{l}) \cap \Omega_{1}(\mathcal{Z}(G))$ must be equal to $\Omega_{1}(\mathcal{Z}(G))$, since $\Omega_{1}(\mathcal{Z}(G))\cong C_{p}^{m+n}$. Hence, $\ker(\chi_l) = \Omega_{1}(\mathcal{Z}(G))$, which implies that $\cap_{i=1}^{m+n}\ker(\chi_i) \neq 1$, a contradiction. Therefore,
   \begin{equation} \label{eq:pgroupHtimeK}
       \ker(\chi_j) \cap \Omega_{1}(\mathcal{Z}(G)) = S_j = \langle t_i : 1\leq i \leq m+n, ~i\neq j \rangle, \text{ for each } j.
   \end{equation}
   Note that $\langle t_i : 1\leq i \leq m+n \rangle \cong C_{p}^{m+n}$. Since $\langle t_i : 1\leq i \leq m+n \rangle$ is a subgroup of $\Omega_{1}(\mathcal{Z}(G))$, and $\Omega_{1}(\mathcal{Z}(G)) \cong C_{p}^{m+n}$, we must have 
   \[ \Omega_{1}(\mathcal{Z}(G))  = \langle t_i : 1\leq i \leq m+n \rangle. \]
   Since $\Omega_{1}(\mathcal{Z}(G)) = \Omega_{1}(\mathcal{Z}(H)) \times \Omega_{1}(\mathcal{Z}(K))$,  we can rearrange $\ker(\chi_i)$ so that 
    \[ \langle t_i : m+1\leq i \leq m+n \rangle \cap \Omega_{1}(\mathcal{Z}(H)) = 1, \text{ and } \langle t_i : 1\leq i \leq m \rangle \cap \Omega_{1}(\mathcal{Z}(K)) = 1, \]
    owing to the fact that $d(\Omega_{1}(\mathcal{Z}(H))) = m$ and $d(\Omega_{1}(\mathcal{Z}(K))) = n$.
    From Equation \eqref{eq:pgroupHtimeK}, 
    \[ \bigcap_{i=1}^{m} \ker(\chi_i) \cap \Omega_{1}(\mathcal{Z}(G)) = \langle t_{i} : m+1\leq i \leq m+n \rangle, \text{ and } \bigcap_{i=m+1}^{m+n} \ker(\chi_i) \cap \Omega_{1}(\mathcal{Z}(G)) = \langle t_{i} : 1\leq i \leq m \rangle. \]
    Hence, we get
    \[ \bigcap_{i=1}^{m} \ker(\chi_i) \cap \Omega_{1}(\mathcal{Z}(G)) \cap \Omega_{1}(\mathcal{Z}(H)) = \langle t_{i} : m+1\leq i \leq m+n \rangle \cap \Omega_{1}(\mathcal{Z}(H)) = 1.  \]
    Therefore, 
    \[ \bigcap_{i=1}^{m} \ker(\chi_i) \cap \Omega_{1}(\mathcal{Z}(H)) = 1 \Rightarrow \bigcap_{i=1}^{m} \ker(\chi_i) \cap \mathcal{Z}(H) = 1. \]
    Similarly, we get 
    \[ \bigcap_{i=m+1}^{m+n} \ker(\chi_i) \cap \mathcal{Z}(K) = 1. \]
    Since $\chi_i$ (for $1\leq i \leq m$) and $\chi_i$ (for $m+1\leq i \leq m+n$), respectively, can be thought of as characters of $H$ and $K$, we get
     $\sum_{i=1}^{m} \chi_{i}(1) \geq \delta(H)$ and $\sum_{i=m+1}^{m+n} \chi_{i}(1) \geq \delta(K)$. Then,
    \[ \delta(G) = \chi(1) = \sum_{i=1}^{n} \chi_i(1) = \sum_{i=1}^{m} \chi_i(1) + \sum_{i=m+1}^{m+n} \chi_i(1) \geq \delta(H) + \delta(K). \]
    Therefore, from Lemma \ref{lemma:deltaHtimesA}, $\delta(G) = \delta(H) + \delta(K)$.
\end{proof}

As a consequence of Theorem \ref{thm:HtimesKpgroup}, if $G = H_1\times H_2 \times \cdots \times H_r$, where each $H_i$ is a $p$-group for some fixed prime $p$, then the representation dimension of $G$ is equal to the sum of the representation dimension of $H_i$'s. However, the same may not be true if $G$ is a nilpotent group. 
In fact, if $H$ and $K$ are non-abelian groups of order $p^n$ and $q^m$, respectively, where $p\neq q$, and center of any one of $H$ and $K$ is non-cyclic, then we may have $\delta(H\times K) \neq \delta(H)+\delta(K)$.


\begin{example} \label{example:nilpotentcounterexample1}
    \textnormal{Let \[ H = \langle x, y, z: x^{4} = y^2 = z^2 = 1, yx = xy, zxz^{-1} = xy, yz = zy \rangle, \] a non-abelian group of order $16$, and \[ K = \langle a, b: a^{9} = b^3 = 1, a^{-1}b^{-1}ab = b^3 \rangle, \] a non-abelian group of order $27$. Here $\mathcal{Z}(H) = \langle x^2, y \rangle \cong C_2 \times C_2$. By \cite[Theorem 3.1(2)]{KKS23}, $\delta(H) = 3$, and $\delta(K) = 3$. However, by using {\sc GAP} \cite{GAP}, we obtain $\delta(H\times K) = 5$. Hence $\delta(H\times K) < \delta(H) + \delta(K)$.}
\end{example}

Now, we prove that $\delta(H\times K) = \delta(H) + \delta(K)$ if $K$ is cyclic (not necessarily of prime-power order), and $H$ belongs to some class of non-abelian $p$-groups with cyclic center.

\begin{theorem} \label{thm:HtimesA1}
    Let $H$ be a non-abelian $p$-group with cyclic center and $K$ be a cyclic group. If $\cd(H) = \{ 1, p^a\}$, or $\{ 1, p, p^b\}$ $(a\geq 1, b > 1)$, then
    \[ \delta(H \times K) = \delta(H) + \delta(K) = \max \cd(H) + 1. \]
\end{theorem}

\begin{proof} Suppose $G = H\times K$. Note that $\delta(K) = 1$, and from \cite[Theorem 3.3]{KKS23}, $\delta(H) = \max \cd(H)$. 
Let $\mathcal{X}\subset \Irr(G)$ affording $\delta(G)$, and $\psi = \sum_{\chi \in \mathcal{X}} \chi$ be a faithful character of $G$ of minimal degree. 
We have the following two cases.
\begin{enumerate}
    \item [(i)] $\cd(H) = \{ 1, p^a \}$ $(a \geq 1):$ Since $G'$ is contained in the kernel of each linear character of $G$, the intersection of the kernels of any collection of linear characters is always non-trivial. Hence, the set $\mathcal{X}$ must contain at least one nonlinear irreducible character of $G$. Since $\cd(H) = \{ 1, p^a \}$, we get $\psi(1) = a_1\cdot 1 + a_2 \cdot p^a$, for some $a_1 \geq 0$ and $a_2 > 0$. 
     However, from Lemma \ref{lemma:deltaHtimesA}, we get 
    \begin{equation*}
       a_1\cdot 1 + a_2 \cdot p^a = \psi(1) = \delta(G) \leq 1 + p^a, \text{ where } a_2>0,
    \end{equation*}
    which implies that $a_1 = 1$ and $a_2 = 1$.
    \item [(ii)] $\cd(H) = \{ 1, p, p^b \}$ $(b > 1):$ Suppose $\rho$ is an irreducible character of $H$ of degree $p$. Since $p$-groups are monomial, there exists a subgroup $M$ of index $p$ in $H$ such that $\rho = \eta\ind_{M}^{H}$, for some $\eta \in \lin(M)$. From \cite[Lemma 12.11]{IBook}, $M$ must be non-abelian. Since $M'$ is a normal subgroup of $H'$, and $H'$ is a characteristic subgroup of $H$, $1 \neq M'$ is a normal subgroup of $H$. Let $x\in \mathcal{Z}(H)$ be an element of order $p$. Since $\mathcal{Z}(H)$ is cyclic, we have $1\neq x\in \mathcal{Z(H)} \cap M' \subseteq M' \subseteq \ker(\eta)$, and thus, $x \in \ker(\rho)$. Hence, $x$ is contained in the kernel of every irreducible character of $H$ of degree $p$. Further, $x \in \mathcal{Z(H)} \cap H' \subseteq H' \subseteq G'$ is contained in the kernel of every linear character of $G$.
   Thus, in case all the  irreducible characters in $\mathcal{X}$ are of degree 1 or $p$,
   then $\cap_{\chi \in \mathcal{X}} \ker(\chi)$ cannot be trivial. Hence, 
   $\psi(1) = a_1\cdot 1 + a_2 \cdot p + a_3 \cdot p^b$, for some $a_1, a_2 \geq 0$ and $a_3 > 0$. 
     However, Lemma \ref{lemma:deltaHtimesA} yields  
    \begin{equation*}
       a_1\cdot 1 + a_2 \cdot p + a_3 \cdot p^b = \psi(1) = \delta(G) \leq 1 + p^a,
    \end{equation*}
    which implies that $a_1 = 1$, $a_2 = 0$ and $a_3 = 1$.
\end{enumerate}
This completes the proof.
\end{proof}

\begin{theorem} \label{thm:HtimesA2}
    Let $H$ be a normally monomial $p$-group with cyclic center and $K$ be a cyclic group. Then, 
\[ \delta(H\times K) = \delta(H) + \delta(K) = \max\cd(H) + 1. \]
\end{theorem} 
\begin{proof}
    Let $G = H\times K$ and suppose $\max \cd(H) = p^b$. Since $\cd(G) = \cd(H)$, we get that $p^b = \max\cd(G)$. Note that $\delta(K) = 1$, and by \cite[Remark 3.2]{KKS23}, $\delta(H) = \max \cd(H) = p^b$. From Lemma \ref{lemma:deltaHtimesA}, $\delta(G) \leq p^b + 1$.
    Let $\mathcal{X}\subset \Irr(G)$ affording $\delta(G)$, and $\psi = \sum_{\chi \in \mathcal{X}} \chi$ be a faithful character of $G$ of minimal degree. To prove the desired result, it is enough to show that there is an irreducible character of $G$ of degree $p^b$ in $\mathcal{X}$.\\
  Suppose, on the contrary, that $\mathcal{X}\cap \nl_{p^b}(G) = \emptyset$, where $\nl_{p^b}(G)$ denotes the set of all irreducible characters of $G$ of degree $p^b$. Then, $\mathcal{X}$ contains an irreducible character, say $\chi$, of $G$ of degree $p^a$, for some fixed $a<b$.
   Since $G$ is a cyclic extension of a normally monomial group, it itself is normally monomial. Hence, 
    \[ \chi = (\lambda_1 \lambda_2)\ind_{N}^{G}, \text{ for some } N = M\times K \trianglelefteq G, \text{ where } M \trianglelefteq H, \]
    and for some $\lambda_1 \in \lin(M)$ and $\lambda_2 \in \lin(K)$. Here $|M| = |H|/p^{a} > |H|/p^{b}$. Since $H$ is normally monomial, from \cite[Proposition 3]{M99}, maximum order of an abelian normal subgroup is $|H|/p^{b}$. Hence, $M$ must be a non-abelian normal subgroup of $H$. 
    Consequently, $M'$ is normal in $H$. Thus, $M'\cap \mathcal{Z}(H) \neq 1$. Let $x$ be an arbitrary element of order $p$ in $M' \cap \mathcal{Z}(H)$. Then, $1\neq x \in M'\cap \mathcal{Z}(H) \subset M' \subset \ker(\lambda_1)$. 
    This implies that
    \[ (x, 1) \in \ker(\lambda_1 \lambda_2) \Rightarrow (x, 1) \in \Core_{G}\left(\ker(\lambda_1 \lambda_2) \right) = \ker(\chi). \]
    Further, $(x,1) \in M' \times 1 = N' \subset G' \subset \ker(\eta)$, for every $\eta \in \lin(G)$. Therefore, in case all the irreducible characters in $\mathcal{X}$ are of degree 1 or $p^a$,
    then $\cap_{\chi \in \mathcal{X}} \ker(\chi)$ cannot be trivial. Thus, $\mathcal{X}$ must contain at least one irreducible character of degree $p^b$ of $G$. 
   Consequently, Lemma \ref{lemma:deltaHtimesA} implies that we must have
    \begin{equation*}
        \psi(1) = \delta(G) = 1 + p^a.
    \end{equation*}
    This concludes the proof.
\end{proof} 

In Theorem \ref{thm:HtimesA1} and \ref{thm:HtimesA2}, if we consider $K$ to be abelian but not cyclic, then $\delta(H\times K)$ may not be equal to $\delta(H)+\delta(K)$. 

\begin{example} \label{example:nilpotentcounterexample2}
\textnormal{We present a few examples. 
	\begin{enumerate}
		\item {\bf $H$ is a p-group with cyclic center and $K$ is abelian but not cyclic:} Let $H$ be an extraspecial group of order 8 and $K$ be the elementary abelian 3-group of rank 2. Then, $\mathcal{Z}(H) \cong C_2$, $\delta(H) = 2$, and $\delta(K) = 2$. However, by using {\sc GAP}, we obtain $\delta(H\times K) = 3$.
        \item {\bf $H$ is a p-group with non-cyclic center and $K$ is abelian but not cyclic:} Let $H$ be the group of order 16 mentioned in Example \ref{example:nilpotentcounterexample1}, and $K$ be the elementary abelian 3-group of rank 2. Then, $\delta(H) = 3$, and $\delta(K) = 2$. However, by using {\sc GAP}, we obtain $\delta(H\times K) = 3$. 
        \item {\bf $\mathcal{Z}(H)$ is cyclic and $K$ is abelian but not cyclic:} Let $H$ be the dihedral group of order $6$, $K$ be the cyclic group of order $2$, and $G = H\times K$. Note that, $H$ is a normally monomial group. Since $G$ is not a cyclic group, $G$ has no faithful linear character. From the character table of $G$, there exists a faithful irreducible character of degree 2. Therefore, $\delta(G) = 2$, whereas, $\delta(H) = 2$ and $\delta(C_2) = 1$ so that $\delta(H) + \delta(K) = 3$.
	\end{enumerate}
	Hence, in all the above examples we get $\delta(H\times K) < \delta(H) + \delta(K)$.}
\end{example}



In the next section, we study the representation dimension of some specific classes of $p$-groups.

\section{Computation of the representation dimension for various classes of $p$-groups} \label{sec:pgroups}

In this section, we focus our attention on computing the representation dimensions of some well-known classes of $p$-groups, namely, VZ $p$-groups, Camina $p$-groups, and metacyclic $p$-groups.

Recall that a group is called a VZ-group \cite{L09b} if all its nonlinear irreducible characters vanish off the center. If $G$ is a VZ-group, then from \cite{FAM01}, $\cd(G) = \{ 1, |G/\mathcal{Z}(G)|^{1/2} \}$. Readers can see \cite{L09a, L09b} for some interesting group-theoretic properties of VZ-groups. We prove the following.
\begin{prop} \label{prop:VZgroup}
    Let $G$ be a VZ $p$-group. Then, 
\[ \delta(G) = d(\mathcal{Z}(G)) - d(G') + d(G')|G:\mathcal{Z}(G)|^{1/2}. \]
\end{prop}

\begin{proof}
Suppose $d(\mathcal{Z}(G)) =r$ and $d(G') = k$ and let $\mathcal{X}\subset \Irr(G)$ affording $\delta(G)$. By Lemma \ref{lemma:Bardestanietal}, we get $|\mathcal{X}| = r$. Now, from \cite[Lemma 2.5]{U23}, $|\mathcal{X}\cap \lin(G)| \leq r - k$. Since $\cd(G) = \{ 1, |G:Z(G)|^{1/2} \}$, we deduce that $\delta(G) \geq r - k + k|G:\mathcal{Z}(G)|^{1/2}$. Further, by \cite[Theorem 2]{PU23b}, there exists a set $X$ containing $k$ nonlinear irreducible characters and $r-k$ linear characters of $G$ that satisfies Equation \eqref{eq:SetXProperty}. This implies that $\delta(G) \leq r - k + k|G:\mathcal{Z}(G)|^{1/2}$. Hence, the result follows.
\end{proof}

Now, a group $G$ is called a Camina group, if $\chi(g) = 0$ for all $\chi\in \nl(G)$ and for all $g\in G \setminus G{}'$. 
The concept of Camina groups emerges from a generalization of Frobenius groups, and it has been studied by several researchers (\cite{DS96, IL15, L14}). 
Camina $p$-groups have nilpotency class at most $3$ (see \cite{DS96}). If $G$ is a Camina $p$-group of class 2, then $G$ is a VZ $p$-group. Otherwise, if $G$ is a Camina $p$-group of class 3, then $|G/G'|=p^{2n}, |G'/\mathcal{Z}(G)|=p^{n}$ and $|G/\mathcal{Z}(G)|=p^{3n},$ where $n$ is even, and $\cd(G) = \{1, p^{n}, p^{3n/2}\}$ (see \cite{M81, PS14}). We prove the following. 
\begin{prop} \label{prop:Caminapgroup}
    Let $G$ be a Camina $p$-group. Then, $\delta(G) = d(\mathcal{Z}(G))|G:\mathcal{Z}(G)|^{1/2}$.
\end{prop}
 
\begin{proof}
    Let $d(\mathcal{Z}(G)) = r$ and let $\mathcal{X}\subset \Irr(G)$ affording $\delta(G)$. From Lemma \ref{lemma:Bardestanietal}, we get $|\mathcal{X}| = r$. Now, by \cite[Lemma 33]{PU23b}, $\mathcal{X} \subset \Irr(G|\mathcal{Z}(G))$, where $\Irr(G|\mathcal{Z}(G))= \Irr(G) \setminus \Irr(G/ \mathcal{Z}(G))$, and by the discussion in Section 3 of \cite{PS14}, we get that if $\chi\in \Irr(G|\mathcal{Z}(G))$ then $\chi(1) = |G:\mathcal{Z}(G)|^{1/2}$. Hence, for all $\chi \in \mathcal{X}$, $\chi(1) = |G:\mathcal{Z}(G)|^{1/2}$. Therefore, we get
    \[ \delta(G) = r|G:\mathcal{Z}(G)|^{1/2} = d(\mathcal{Z}(G))|G:\mathcal{Z}(G)|^{1/2}, \]
    which concludes the proof.
\end{proof}

Lastly, we compute the representation dimension of metacyclic $p$-groups. A group $G$ is called metacyclic if there exists a cyclic normal subgroup $N$ of $G$ such that $G/N$ is also cyclic. The class of metacyclic $p$-groups is one of the most well-studied in recent times. We obtain the following result.
\begin{prop} \label{prop:metacyclic}
    Let $G$ be a metacyclic $p$-group and 
\begin{enumerate}
    \item [(i)] let $\mathcal{Z}(G)$ be cyclic. If $G$ is split so that $G$ has the form
    \[ \langle a, b~|~ a^{p^{m}} = b^{p^{m-r}} = 1, b^{-1}ab = a^{1+p^{r}} \rangle, \]
    or $G$ is non-split so that $G$ has the form
    \[ \langle a, b~|~ a^{p^{m}} = 1, b^{p^{n}} = a^{p^{m-r}}, b^{-1}ab = a^{1+p^{r}} \rangle, \]
    then $\delta(G) = \delta_{\irr}(G) = p^{m-r}$.
    \item [(ii)] let $\mathcal{Z}(G)$ be non-cyclic so that $G$ has the form
    \[ \langle a, b~|~ a^{p^{m}} = 1, b^{p^{n}} = a^{p^{m-r}}, b^{-1}ab = a^{1+p^{r}} \rangle. \]
    Then, $\delta(G) = 1 + p^{m-r}$.
\end{enumerate}
\end{prop} 

\begin{proof}
    First, assume that $\mathcal{Z}(G)$ is cyclic. Since $G$ is a $p$-group, from \cite[Corollary 4.2]{KKS23}, we have $\delta(G) = \delta_{\irr}(G)$. From the proof of \cite[Theorem 4.30]{B95}, we get that a faithful irreducible character has degree $p^{m-r}$. Hence, the result follows.\\
    Now, assume that $\mathcal{Z}(G)$ is non-cyclic. Then, $d(\mathcal{Z}(G)) = 2$. Let $\chi$ be a faithful character of $G$ such that $\chi(1) = \delta(G)$. From Lemma \ref{lemma:Bardestanietal}, $\chi = \chi_1 + \chi_2$, for some $\chi_1, \chi_2 \in \Irr(G)$. From the proof of \cite[Lemma 3.10]{B00}, we get that $\chi$ is faithful only when $\chi_1(1) = 1$ and $\chi_2(1) = p^{m-r}$. Hence, the result follows.
\end{proof}

Now, we compute the representation dimension of groups of order $p^6$. The groups of order $p^6$ for $p\geq 5$ are classified into 43 different isoclinic families $\Phi_i$, $1\leq i \leq 43$; $\Phi_1$ contains all the abelian groups. Groups of order $2^6$ and $3^6$ have different presentations. Readers can see \cite{HS64,M69} and \cite{Baldwin87} for presentations of groups of order $2^6$ and $3^6$, respectively; they are also included in the {\sc SmallGroups} library \cite{SmallGroups}
in {\sf GAP} \cite{GAP}. Their representation dimensions can be obtained easily by using the {\sf GAP} \cite{GAP} function {\tt EmbeddingDegree} developed by Kaur et al. \cite{KKS23}. \\
Henceforth, $p$ denotes a prime greater than $5$. Newman et al. 
\cite{NO'BV04} established that there are  
\[ 3p^2 + 39 p + 344 + 24\gcd(p-1,3) + 11 \gcd(p-1,4) + 2 \gcd(p-1,5) \]
groups of order $p^{6}$, up to isomorphism.  We urge readers to see \cite{O'BPU24} for the information on group presentations, and \cite[Table 4.1]{J80} for the information on the order of the center, the structure of the derived subgroup, and the set of character degrees corresponding to each isoclinic family. Henceforth in this section, $G$ denotes a non-abelian group of order $p^6$.

The following results help us to list representation dimension of groups of order $p^{6}$, in Table \ref{t:delta G p6}.
\begin{lemma} \label{lemma:degree center p}
    Let $G$ belongs to $\Phi_i$ where $i \in \{ 22, 24, 25, \ldots, 43 \}$. If $G \in \Phi_{35}$, then $\delta(G) = p$, otherwise, $\delta(G) = p^2$.
\end{lemma}
\begin{proof}
    If $G \in \Phi_{35}$, then $\cd(G)=\{1,p\}$, and hence by \cite[Theorem 4.4]{KKS23}, $\delta(G)$ equals $p$. Otherwise, $\cd(G)=\{1,p,p^{2}\}$ (see \cite{O'BPU24}), and hence by \cite[Theorem 4.4]{KKS23}, we get $\delta(G)=p^{2}$.
\end{proof}
We obtain the following Corollary with the help of Proposition \ref{prop:VZgroup}.
\begin{cor} \label{cor: VZgroups degree}
    Let $G$ be a VZ $p$-group.
    \begin{enumerate}
        \item If $G\in \Phi_2$, then $\delta(G) = p$, $p+1$, $p+2$, or $p+3$.
        \item If $G\in \Phi_5$, then $\delta(G) = p^2$, or $p^2+1$.
        \item If $G\in \Phi_{15}$, then $\delta(G) = 2p^2$.
    \end{enumerate}
\end{cor}
\begin{proof}
    Note that if $G\in \Phi_2$, then $|\mathcal{Z}(G)| = p^4$ and $\cd(G) = \{1, p\}$, and if $G\in \Phi_i$ (for $i=5,15$), then $|\mathcal{Z}(G)| = p^2$ and $\cd(G) = \{1, p^2\}$. In all the above cases, $\cd(G) = \{ 1, |G:\mathcal{Z}(G)|^{1/2} \}$. Hence $\Phi_2$, $\Phi_5$ and $\Phi_{15}$ are VZ-group families. From \cite[Table 4.1]{J80}, $d(G') = 1$ if $G\in \Phi_i$ (for $i=2,5$), and $d(G') = 2$ if $G\in \Phi_{15}$. Further, for $G\in \Phi_2$, $d(\mathcal{Z}(G)) = 1,2,3$, or $4$, for $G\in \Phi_5$, $d(\mathcal{Z}(G)) = 1$, or $2$, and for $G\in \Phi_{15}$, $d(\mathcal{Z}(G)) = 2$. By Proposition \ref{prop:VZgroup}, we get $\delta(G) = p$, $p+1$, $p+2$, or $p+3$ when $G\in \Phi_2$, $\delta(G) = p^2$, or $p^2+1$ when $G\in \Phi_5$, and $\delta(G) = 2p^2$ when $G\in \Phi_{15}$. 
\end{proof}

 Further, if $G \in \Phi_{14}$ then $\mathcal{Z}(G) \cong C_{p^2}$ and $\cd(G) = \{ 1,p,p^2 \}$. By \cite[Theorem 4.4]{KKS23}, $\delta(G) = p^2$.
Now, we are left with the families $\Phi_3$, $\Phi_4$, $\Phi_6, \ldots, \Phi_{13}$, $\Phi_{16},\ldots, \Phi_{21}$ and $\Phi_{23}$; the computation of representation dimension for these families are more difficult to handle. We obtain their representation dimension by utilizing the results obtained in \cite{KKS23} and \cite{O'BPU24}. Note that in the families $\Phi_3$, $\Phi_4$, $\Phi_6, \ldots, \Phi_{10}$, there are some groups which can be expressed as a direct product of an abelian group and a non-abelian group of order less than $p^6$. For such groups, the representation dimension can easily be computed by utilizing Theorem \ref{thm:HtimesKpgroup} and \cite{KKS23}. We now focus on those groups that are not a nontrivial direct product of groups. We start with the $\Phi_3$ family.

\begin{itemize}
    \item $\mathbf{\Phi_3}$ {\bf family:} For $G\in \Phi_3$, $G' \cong C_p \times C_p$, $\mathcal{Z}(G) \cong C_{p^3}$, $C_{p^2}\times C_p$, or $C_{p}^3$, and $\cd(G) = \{ 1, p\}$. When the center is cyclic, $\delta(G) = p$ by \cite[Theorem 4.4]{KKS23}. Now, assume that $\mathcal{Z}(G) \cong C_{p^2}\times C_p$. From Lemma \ref{lemma:Bardestanietal}, $\chi = \chi_1+\chi_2$, for some $\chi_1, \chi_2 \in \Irr(G)$. Since neither of $\chi_i$ can be linear, we must have $\chi(1) \geq p+1$. We prove the reverse inequality by considering each group in $\Phi_3$ with $\mathcal{Z}(G) \cong C_{p^2} \times C_p$ separately. For example, take $G = G_{(3, 10r)}$ (for $r=1$ or $\nu$, where $\nu$ is the smallest positive integer which is a quadratic non-residue modulo $p$). By \cite[Table 5]{O'BPU24}, there exists $\mathcal{X}\subset \Irr(G_{(3, 10r)})$ satisfying Equation \eqref{eq:SetXProperty} and $\sum_{\rho \in \mathcal{X}}\rho(1) = p+1$. Then, $\delta(G_{(3, 10r)}) \leq p+1$. Therefore, $\delta(G_{(3, 10r)}) = p+1$. Along similar lines, we prove that for all $G\in \Phi_3$ with $\mathcal{Z}(G) \cong C_{p^2} \times C_p$, we get $\delta(G) = p+1$.\\
    Now, we consider the case when $\mathcal{Z}(G) \cong C_{p}^3$. From Lemma \ref{lemma:Bardestanietal}, $\chi = \sum_{i=1}^3\chi_i$, for some $\chi_1, \chi_2, \chi_3 \in \Irr(G)$. 
    Clearly, $\delta(G) \geq p+2$. We prove the reverse inequality by considering each group in $\Phi_3$ with $\mathcal{Z}(G) \cong C_{p}^3$ separately, as above. We get that for all $G\in \Phi_3$ with $\mathcal{Z}(G) \cong C_p^3$, $\delta(G) = p+2$.
    \item $\mathbf{\Phi_{4}}$ and $\mathbf{\Phi_{6}}$ {\bf families:} For $G\in \Phi_{4}$, $G' \cong C_{p}\times C_{p}$, $\mathcal{Z}(G) \cong C_{p^2}\times C_p$,  or $C_{p}^3$, $G'\cap \mathcal{Z}(G) \cong C_p \times C_p$ and $\cd(G) = \{ 1, p \}$. Suppose $\mathcal{Z}(G) \cong C_{p^2}\times C_p$. Here, $d(G' \cap \mathcal{Z}(G)) = d(\mathcal{Z}(G))$, and hence $G'$ contains all the $p$-ordered elements of $\mathcal{Z}(G)$. Let $\chi$ be a faithful character of $G$ such that $\chi(1) = \delta(G)$.  From Lemma \ref{lemma:Bardestanietal}, $\chi = \chi_1+\chi_2$, for some $\chi_1, \chi_2 \in \Irr(G)$. It is easy to see that if either of  $\chi_1$ and $\chi_2$ is linear, then $\ker(\chi_1) \cap \ker(\chi_2) \neq 1$. Thus, $\delta(G) = \chi(1)$ must be $2p$.\\
    Now, suppose $\mathcal{Z}(G) \cong C_{p}^3$. From Lemma \ref{lemma:Bardestanietal}, $\chi = \sum_{i=1}^3\chi_i$, for some $\chi_1, \chi_2, \chi_3 \in \Irr(G)$. Note that here $G' \subset \mathcal{Z}(G)$. Since at most two of $\chi_i$ can be linear, assume, without loss of generality, that $\chi_1$ and $\chi_2$ are linear. Now, $\cap_{i=1, i\neq j}^3 \ker(\chi_i) \neq 1$ for each $j=1,2,3$. Hence, $\cap_{i=1, i\neq j}^3 (\ker(\chi_i) \cap \mathcal{Z}(G)) \neq 1$ for each $j=1,2,3$. However, $\ker(\chi_1) \cap \mathcal{Z}(G) = G' = \ker(\chi_2) \cap \mathcal{Z}(G)$, which implies that $(\ker(\chi_1) \cap \mathcal{Z}(G)) \cap \ker(\chi_3) \cap \mathcal{Z}(G) = 1$, which is a contradiction. Therefore, at most one of $\chi_1$ ($i=1,2,3$) can be linear. Thus, $\chi(1) \geq 2p+1$.\\
    We prove the reverse inequality by considering each group in $\Phi_4$ with an elementary abelian center separately. For example, take $G = G_{(4, 28)}$. By \cite[Table 5]{O'BPU24}, there exists $\mathcal{X}\subset \Irr(G_{(4, 28)})$ satisfying Equation \eqref{eq:SetXProperty} and $\sum_{\rho \in \mathcal{X}}\rho(1) = 2p+1$. Then, $\delta(G_{(4, 28)}) \leq 2p+1$. Therefore, $\delta(G_{(4, 28)}) = 2p+1$. Along similar lines, we have computed the representation dimension for the rest of the groups in $\Phi_{4}$, and we conclude that for all  $G \in \Phi_{4}$, $\delta(G) = 2p$ when $\mathcal{Z}(G) \cong C_{p^2} \times C_p$, and $\delta(G) = 2p+1$  when $\mathcal{Z}(G) \cong C_{p}^3$.\\
    By similar arguments, we obtain that for all $G \in \Phi_{6}$, $\delta(G) = 2p$ when $\mathcal{Z}(G) \cong C_{p^2} \times C_p$, and $\delta(G) = 2p+1$  when $\mathcal{Z}(G) \cong C_{p}^3$.
    \item $\mathbf{\Phi_{7}}$, $\mathbf{\Phi_{8}}$ and $\mathbf{\Phi_{10}}$ {\bf families:} For $G\in \Phi_{7}$, $G' \cong C_{p}\times C_{p}$, $\mathcal{Z}(G) \cong C_{p^2}$,  or $C_p \times C_p$, $G'\cap \mathcal{Z}(G) \cong C_p$ and $\cd(G) = \{ 1, p, p^2 \}$. If $\mathcal{Z}(G)$ is cyclic, then $\delta(G) = p^2$ by \cite[Theorem 4.4]{KKS23}. Now, suppose $\mathcal{Z}(G) \cong C_p \times C_p$. From Lemma \ref{lemma:Bardestanietal}, $\chi = \chi_1+\chi_2$, for some $\chi_1, \chi_2 \in \Irr(G)$. 
    Since both the characters cannot be linear, we must have $\chi(1) \geq p+1$. On the other hand, if $\eta$ is a nonlinear character of degree $p$ of $G$, then $\eta = \psi\ind_{H}^{G}$, for some linear character $\psi$ of some subgroup $H$ of index $p$ in $G$. Note that $H$ is a normal subgroup of $G$. Since $H'$ is a characteristic subgroup of $H$, we get that $H'$ is normal in $G$. Hence, $H'\cap \mathcal{Z}(G) \neq 1$. In fact, we get $H'\cap \mathcal{Z}(G) = G'\cap \mathcal{Z}(G) \subset \Core_{G}(\ker(\psi)) = \ker(\eta)$. Hence, if one of $\chi_1$ and $\chi_2$ is linear, then the other cannot be of degree $p$. Further, both $\chi_1$ and $\chi_2$ cannot be simultaneously of degree $p$. Thus, $\delta(G) \geq p^2+1$. \\
    We prove the reverse inequality by considering each group in the family separately. For example, take $G = G_{(7, 14)}$. By \cite[Table 6]{O'BPU24}, there exists $\mathcal{X}\subset \Irr(G_{(7, 14)})$ satisfying Equation \eqref{eq:SetXProperty} and $\sum_{\rho \in \mathcal{X}}\rho(1) = p^2+1$. Then, $\delta(G_{(7, 14)}) \leq p^2+1$. Therefore, $\delta(G_{(7, 14)}) = p^2+1$. Along similar lines, we have computed the representation dimension for the rest of the groups in $\Phi_{7}$, and we conclude that $\delta(G) = p^2+1$ for all  $G \in \Phi_{7}$ when the center is non-cyclic.\\
    By similar arguments, we obtain that for each $G \in \Phi_{i}$ where $i=8,10$, $\delta(G) = p^2$ when $\mathcal{Z}(G)$ is cyclic, and $\delta(G) = p^2+1$ when $\mathcal{Z}(G)$ is non-cyclic.
     \item $\mathbf{\Phi_{9}}$ {\bf family:} For $G\in \Phi_{9}$, $G' \cong C_{p}^3$, $\mathcal{Z}(G) \cong C_{p^2}$, or $C_p \times C_p$, $G'\cap \mathcal{Z}(G) \cong C_p$ and $\cd(G) = \{ 1, p\}$. If $\mathcal{Z}(G)$ is cyclic, then $\delta(G) = p$ by \cite[Theorem 4.4]{KKS23}. Now, suppose $\mathcal{Z}(G) \cong C_p \times C_p$. From Lemma \ref{lemma:Bardestanietal}, $\chi = \chi_1+\chi_2$, for some $\chi_1, \chi_2 \in \Irr(G)$. Since both the characters cannot be linear, we must have $\chi(1) \geq p+1$. On the other hand, if $\eta \in \nl(G)$ then $\eta = \psi\ind_{H}^{G}$, for some linear character $\psi$ of some subgroup $H$ of index $p$ in $G$. Note that $H$ is a normal subgroup of $G$. Since $H'$ is a characteristic subgroup of $H$, we get that $H'$ is normal in $G$. Hence, $H'\cap \mathcal{Z}(G) \neq 1$. In fact, we get $H'\cap \mathcal{Z}(G) = G'\cap \mathcal{Z}(G) \subset \Core_{G}(\ker(\psi)) = \ker(\eta)$, for each $\eta \in \nl(G)$. Hence, at most one of $\chi_1$ and $\chi_2$ can be of degree $p$. Hence, $\delta(G) \leq p+1$, and thus, we get $\delta(G) = p+1$.
    \item $\mathbf{\Phi_{11}}$ {\bf family:} For $G\in \Phi_{11}$, $G' = \mathcal{Z}(G) \cong C_p^3$, and $\cd(G) = \{ 1, p\}$. Let $\chi$ be a faithful character of $G$ such that $\chi(1) = \delta(G)$. From Lemma \ref{lemma:Bardestanietal}, $\chi = \sum_{i=1}^3\chi_i$, for some $\chi_1, \chi_2, \chi_3 \in \Irr(G)$. It is easy to see that if $\chi_i \in \lin(G)$ for any $i$ ($i=1,2,3$), then $\cap_{i=1}^3 \ker(\chi_i) \neq 1$, which is a contradiction. Hence, $\chi_i(1) = p$ for each $i$, and hence, $\delta(G) = \chi(1) = 3p$.
    \item $\mathbf{\Phi_{12}}$ and $\mathbf{\Phi_{19}}$ {\bf families:} For $G\in \Phi_{12}$, $G' = \mathcal{Z}(G) \cong C_p \times C_p$, and $\cd(G) = \{ 1, p, p^2\}$. Let $\chi$ be a faithful character of $G$ such that $\chi(1) = \delta(G)$. From Lemma \ref{lemma:Bardestanietal}, $\chi = \chi_1+\chi_2$, for some $\chi_1, \chi_2 \in \Irr(G)$. It is easy to see that if $\chi_i \in \lin(G)$ for any $i$ ($i=1,2$), then $\cap_{i=1}^2 \ker(\chi_i) \neq 1$, which is a contradiction. Hence, $\chi_i(1) \geq p$ for each $i$, and hence, $\delta(G) \geq 2p$. On the other hand, building on the arguments of the proof of Lemma 4.21 in \cite{O'BPU24}, we obtain a set of two nonlinear irreducible characters of degree $p$ of $G$ such that the intersection of their kernels is trivial. Hence, $\delta(G)\leq 2p$. Finally, we get $\delta(G) = 2p$.\\
    By similar arguments, we obtain $\delta(G) = 2p$ for each $G \in \Phi_{19}$.
    \item $\mathbf{\Phi_{13}}$, $\mathbf{\Phi_{18}}$ and $\mathbf{\Phi_{20}}$ {\bf families:} For $G\in \Phi_{13}$, $G' = \mathcal{Z}(G) \cong C_p \times C_p$, and $\cd(G) = \{ 1, p, p^2\}$. Let $\chi$ be a faithful character of $G$ such that $\chi(1) = \delta(G)$. From Lemma \ref{lemma:Bardestanietal}, $\chi = \chi_1+\chi_2$, for some $\chi_1, \chi_2 \in \Irr(G)$. By the proof of Lemma 4.22 of \cite{O'BPU24}, $\chi_i \in \nl(G)$ for both $i$ and $\chi_i(1) = p$ for at most one $i$ where $i=1,2$. Hence, $\delta(G) \geq p^2 + p$. We prove the reverse inequality by considering each group in the family separately. For example, take $G = G_{(13, 1)}$. By \cite[Table 6]{O'BPU24}, there exists $\mathcal{X}\subset \Irr(G_{(13, 1)})$ satisfying Equation \eqref{eq:SetXProperty} and $\sum_{\rho \in \mathcal{X}}\rho(1) = p^2 + p$. Then, $\delta(G_{(13, 1)}) \leq p^2 + p$. Therefore, $\delta(G_{(13, 1)}) = p^2 + p$. Along similar lines, we have computed the representation dimension for the rest of the groups in $\Phi_{13}$, and we conclude that $\delta(G) = p^2+p$ for all  $G \in \Phi_{13}$.\\
    By similar arguments, we obtain $\delta(G) = p^2+p$ for each $G \in \Phi_{i}$, where $i=18,20$.
    \item $\mathbf{\Phi_{16}}$ {\bf family:} For $G\in \Phi_{16}$, $\mathcal{Z}(G) \subset G'$, and $\cd(G) = \{ 1, p\}$. This implies that $\mathcal{Z}(G) \subset \ker(\eta)$ for each $\eta \in \lin(G)$. If $\mathcal{X}\subset \Irr(G)$ affords $\delta(G)$, it is easy to see that $\mathcal{X}\cap \lin(G) = \emptyset$. Further, $d(\mathcal{Z}(G)) = 2$, which implies that $\delta(G) = 2p$.
    \item $\mathbf{\Phi_{17}}$ and $\mathbf{\Phi_{23}}$ {\bf families:} For $G\in \Phi_{17}$, $G'\cong C_{p}^3$, $\mathcal{Z}(G)\cong C_p \times C_p$, $\mathcal{Z}(G) \subset G'$ and $\cd(G) = \{ 1,p,p^2 \}$. Let $\chi$ be a faithful character of $G$ such that $\chi(1) = \delta(G)$. From Lemma \ref{lemma:Bardestanietal}, $\chi = \chi_1+\chi_2$, for some $\chi_1, \chi_2 \in \Irr(G)$. Clearly, $\delta(G)\geq 2p$. 
    We prove the reverse inequality by considering each group in the family separately. For example, take $G = G_{(17, 20)}$. By \cite[Table 6]{O'BPU24}, there exists $\mathcal{X}\subset \Irr(G_{(17, 20)})$ satisfying Equation \eqref{eq:SetXProperty} and $\sum_{\rho \in \mathcal{X}}\rho(1) = 2p$. Then, $\delta(G_{(17, 20)}) \leq 2p$. Therefore, $\delta(G_{(17, 20)}) = 2p$. Along similar lines, we have computed the representation dimension for the rest of the groups in $\Phi_{17}$, and we conclude that $\delta(G) = 2p$ for all  $G \in \Phi_{17}$.\\
     By similar arguments, we obtain $\delta(G) = 2p$ for each $G \in \Phi_{23}$.
    \item $\mathbf{\Phi_{21}}$ {\bf family:} For $G\in \Phi_{21}$, $\mathcal{Z}(G)\cong C_p \times C_p$ and $\cd(G) = \{ 1,p,p^2 \}$. Let $\chi$ be a faithful character of $G$ such that $\chi(1) = \delta(G)$. From Lemma \ref{lemma:Bardestanietal}, $\chi = \chi_1+\chi_2$, for some $\chi_1, \chi_2 \in \Irr(G)$. By the proof of Lemma 4.28 of \cite{O'BPU24}, $\chi_i(1) = p^2$ for $i=1,2$. Hence, $\delta(G) = 2p^2$.
    
\end{itemize}

We summarize the above discussion as the following table.

 \begin{longtable}[c]{| l| c |}
		\caption{ Representation dimension of groups of order $p^{6}$, where $p \geq 5$ is a prime, via their isoclinic families  \label{t:delta G p6}}\\

		\hline
		Isoclinic family &  $\delta(G)$  \\ 
		\hline
		\endfirsthead
		
		\hline
		\multicolumn{2}{|c|}{Continuation of Table \ref{t:delta G p6}}\\
		\hline
		Isoclinic family &  $\delta(G)$  \\   
		\hline
		\endhead
		
		\hline
		\endfoot

		\endlastfoot
		\hline
		
		$\Phi_2$ &  $p$, $p+1$, $p+2$, or $p+3$ \\
		\hline
		$\Phi_3$ & $p$, $p+1$, or $p+2$ \\
		\hline
		$\Phi_4$, $\Phi_6$ & $2p$, or $2p+1$ \\
		\hline
		$\Phi_i$, where $i\in \{ 5, 7, 8, 10 \}$ & $p^2$, or $p^2+1$ \\
		\hline 
		$\Phi_9$ & $p$, or $p+1$ \\ 
		\hline
		$\Phi_{11}$ & $3p$ \\
		\hline
		$\Phi_i$, where $i\in \{ 12, 16, 17, 19, 23 \}$ & $2p$\\
		\hline 
		$\Phi_i$, where $i\in \{ 13, 18, 20 \}$ & $p+p^2$ \\
		\hline
		$\Phi_i$, where $i\in \{ 14, 22, 24, 25, \ldots, 43 \} - \{ 35 \}$ & $p^2$ \\
		\hline
		$\Phi_{15}$, $\Phi_{21}$ & $2p^2$ \\
		\hline
		$\Phi_{35}$ & $p$ \\
		\hline

	\end{longtable}
    
We conclude by the observation that in Table \ref{t:delta G p6}, if $G$ and $H$ are isoclinic groups of order $p^6$ such that $\mathcal{Z}(G) \cong \mathcal{Z}(H)$, then $\delta(G) = \delta(H)$. However, we are unable to prove it at this moment.

\section{Problems}

There are a few open problems for the readers. 
\begin{enumerate}
    \item In \cite[Theorem 3.1]{KKS23}, Kaur et al. have computed the representation dimension for the extraspecial $p$-groups, which are in fact, special $p$-groups with cyclic center. Hence, one would like to examine the representation dimension for special $p$-groups.
    \item We observe that if $G$ and $H$ are isoclinic groups of order $p^6$ such that $\mathcal{Z}(G) \cong \mathcal{Z}(H)$, then $\delta(G) = \delta(H)$ (see Table \ref{t:delta G p6}). This motivates us to ask the follwing question.\\
    Let $G$ and $H$ be isoclinic groups of equal order such that $\mathcal{Z}(G) \cong \mathcal{Z}(H)$. Whether $\delta(G) = \delta(H)$?
    
\end{enumerate}

\section*{Declaration of competing interest}
The authors declare that they have no known competing financial interests or personal relationships that could have appeared to influence the work reported in this paper.

\section*{Acknowledgements} 
 \noindent The first-named author acknowledges the research support of the Department of Science and Technology (INSPIRE Faculty No. DST/INSPIRE/04/2023/001200), Govt. of India. A part of this research work was carried out when the third-named author was a postdoctoral fellow at IISER Mohali, India, and he acknowledges IISER Mohali for an Institute postdoctoral fellowship during that time.


\begin{thebibliography}{99}

\bibitem{BMS16} M. Bardestani, K. Mallahi-Karai and H. Salmasian, Minimal dimension of faithful representations for $p$-groups, J. Group Theory, {\bf 19} (04) (2016), 589-608.

\bibitem{B95} H. Behravesh, Quasi-permutation representations of finite groups, Ph.D. Thesis, University
of Manchester (1995).

\bibitem{B00} H. Behravesh, Quasi-permutation Representations of metacyclic $p$-groups with non-cyclic center, Southeast Asian Bull. Math., {\bf 24} (2000), 345--353.


\bibitem{BZbook} Y. G. Berkovich and E. M. Zhmud’, Characters of finite groups, Part 1. American
Mathematical Soc. (1998).

\bibitem{CKR11} S. Cernele, M. Kamgarpour and Z. Reichstein, Maximal representation dimension of finite $p$-groups, J. Group Theory, {\bf 14} (04) (2011), 637--647.

\bibitem{DS96} R. Dark and C. M. Scoppola, On Camina group of prime power order, J. Algebra, {\bf 181}(3) (1996), 787--802.

\bibitem{FAM01} G. A. Fern\'andez-Alcober and A. Moret\'o, Groups with two extreme character degrees and their normal subgroups, Trans. Amer. Math. Soc., {\bf 353}(6) (2001), 2171--2192.

\bibitem{GAP} The GAP Group, Gap – groups, algorithms, and programming, version 4.13.1, 2024.

\bibitem{HS64} Marshall Hall, Jr., and James K.\ Senior, 
{\it The Groups of Order $2^n$ ($n \leq 6$)}, Macmillan, New York, 1964.

\bibitem{Baldwin87} D.Baldwin, The Groups of order $3^{n}$; for $n^{6}$, B.Sc. Thesis, Australian National University, 1987.

\bibitem{B69} B.G. Basmaji,  Monomial Representations and Metabelian Groups, Nagoya Math. J., {\bf 35} (1969), 99-107.

\bibitem{SmallGroups} H.U.\ Besche, B.\ Eick and E.A.O'Brien, A millenium project: constructing small groups, Internat. J. Algebra Comput., {\bf 12}{2002}, 623--644.

\bibitem{H84} G. How, Special Classes of Monomial Groups III, Chinese J. Math., {\bf 12}(3) (1984), 199--211.

\bibitem{IBook} I. M. Isaacs, Character theory of finite groups, Academic, New York, 1976.

\bibitem{IL15} I. M. Isaacs, M. L. Lewis, Camina p-groups that are generalized Frobenius complements, Arch. Math. (Basel), {\bf 104}(5) (2015), 401--405.

\bibitem{J80} R. James, The groups of order $p^6$ ($p$ an odd prime), Math. Comp., {\bf 34}(150) (1980), 613--637.

\bibitem{KM08} N. A. Karpenko and A. S. Merkurjev, Essential dimension of finite $p$-groups, Invent. Math., {\bf 172} (2008), 491--508.

\bibitem{KKS23} G. Kaur, A. Kulshrestha, A. Singh, Representation dimension of some finite groups, (2026) \href{https://arxiv.org/abs/2308.01612}{arXiv:2308.01612 [math.GR]} 

\bibitem{L14}  M. L. Lewis, Classifying Camina groups: a theorem of Dark and Scoppola. Rocky Mountain J. Math., {\bf 44}(2) (2014), 591--597.

\bibitem{L09a} M. L. Lewis, The vanishing-off subgroup, J. Algebra, {\bf 321}(4) (2009), 1313--1325.
		
\bibitem{L09b} M. L. Lewis, Character tables of groups where all nonlinear irreducible characters vanish off the center, Ischia group theory 2008 (2009), 174--182.

\bibitem{M81} I. D. Macdonald, Some p-groups of Frobenius and extra-special type, Isr. J. Math., {\bf 40} (1981), 350--364.

\bibitem{M99} A. Mann, Minimal Characters of $p$-groups, J. Group Theory,  {\bf 02} (1999), 225--250.

\bibitem{M69}
John McKay,  Table errata: The groups of order $2^n (n \leq 6)$ 
(Macmillan, New York, 1964) by M.\ Hall, Jr.\ and J.K.\ Senior, 
Math.\ Comp.\ {\bf 23} (1969), no.\ 107, 691--692.

\bibitem{MR10} A. Meyer and Z. Reichstein, Some consequences of the Karpenko-Merkurjev theorem, Documenta Math., Extra Volume dedicated to Andrei A. Suslin’s Sixtieth Birthday (2010), 445--457.

\bibitem{M21} A. Moret\'{o}, On the minimal dimension of a faithful linear representation of a finite group (2021) \href{https://arxiv.org/abs/2102.01463}{arXiv:2102.01463v3 [math.GR]}

\bibitem{NO'BV04} M.F. Newman, E.A. O'Brien and M.R. Vaughan-Lee, Groups and nilpotent Lie rings whose order is the sixth power of a prime, J. Algebra, {\bf 278} (2004), 383--401.

\bibitem{NO'BV23} M.F. Newman, E.A. O'Brien and M.R. Vaughan-Lee, 
Presentations for the groups of order $p^6$ for prime $p \geq 7$ (2023)
\href{http://arxiv.org/abs/2302.02677}{arXiv:2302.02677 [math.GR]}

\bibitem{O'BPU24} E. A. O'Brien, S. K. Prajapati, and A. Udeep. Minimal degrees for faithful permutation representations of groups of order $p^6$ where $p$ is an odd prime, J. Algebraic Combin., {\bf 60}(02) (2024), 319--388.

\bibitem{PS14} S. K. Prajapati and B. Sury, On the total character of finite groups, Int. J. Group Theory, {\bf 03}(3) (2014), 47--67.

\bibitem{PU23a} S. K. Prajapati and A. Udeep, Minimal Faithful Quasi-Permutation Representation Degree of $p$-Groups with Cyclic Center, Proc. Indian Acad. Sci. Math. Sci., {\bf 133}(38) (2023).

\bibitem{PU23b} S. K. Prajapati and A. Udeep, On faithful quasi-permutation representation of VZ groups and Camina $p$-groups, Comm. Algebra, {\bf 51}(4) (2023), 1431--1446.

\bibitem{TZ00} P. H. Tiep and A. E. Zalesskii, Some aspects of finite linear groups: A survey, J. Math. Sci., {\bf 100} (2000), 1893--1914.

\bibitem{U23} A. Udeep, On the Minimal Faithful Quasi-permutation Representation Degree of Finite $p$-groups, Indian Institute of Technology Bhubaneswar, Ph.D. Thesis (2023).

\bibitem{W75} D. Wright, Degrees of Minimal Embeddings for Some Direct Products, Amer. J. Math., {\bf 97}(04) (1975), 897--903.
    
\end{thebibliography}
\end{document}